\documentclass{article}
\usepackage[all]{xy}
\xyoption{line}
\usepackage{graphicx}

\usepackage{amsmath, amsfonts, amsthm, verbatim}

\theoremstyle{plain}
\newtheorem{thm}{Theorem}[section]
\newtheorem{lem}[thm]{Lemma}
\newtheorem{cor}[thm]{Corollary}

\theoremstyle{definition}
\newtheorem{defn}[thm]{Definition}

\newtheorem{exl}[thm]{Example}
\newtheorem*{exlctd}{Example \ref{picexl}, Continued}

\newcommand{\R}{\mathbb{R}}
\newcommand{\Z}{\mathbb{Z}}

\newcommand{\lift}{\widetilde}
\newcommand{\htp}{\simeq}
\newcommand{\MC}{\text{\textit{MC}}}

\newcommand{\ME}{\text{\textit{ME}}}
\newcommand{\Reid}{\mathcal{R}}

\DeclareMathOperator{\Coin}{Coin}
\DeclareMathOperator{\Eq}{Eq}
\DeclareMathOperator{\ind}{ind}
\DeclareMathOperator{\sign}{sign}

\DeclareMathOperator{\im}{im}
\DeclareMathOperator{\coker}{coker}

\title{Nielsen equalizer theory}
\author{P. Christopher Staecker\thanks{
Address: Department of Mathematics and Computer Science, Fairfield
University, Fairfield CT, USA}
\thanks{Email: cstaecker@fairfield.edu}
\thanks{Web: {http://faculty.fairfield.edu/cstaecker}}
}

\begin{document}
\bibliographystyle{hplain}
\maketitle

\begin{abstract}
We extend Nielsen coincidence theory to equalizer sets,
the points where a given set of (more than 2) mappings agree. On
manifolds, this theory is interesting only for maps between spaces of
different dimension, and our results hold for sets of $k$ maps on
compact manifolds from dimension $(k-1)n$ to dimension $n$. We define
the Nielsen equalizer number, which is a lower bound for the minimal
number of equalizer points when the maps are changed by
homotopies, and is in fact equal to this minimal number when the
domain manifold is not a surface.

As an application we give some results in Nielsen coincidence theory
with positive codimension. This includes a complete computation of the
geometric Nielsen number for maps between tori.
\end{abstract}

\section{Introduction}
The goal of this paper is to generalize the basic definitions and results of Nielsen coincidence theory to a theory of equalizer sets for sets of (possibly more than two) mappings. For spaces $X$ and $Y$ and maps $f_1, \dots, f_k:X \to Y$, the equalizer set is defined as
\[ \Eq(f_1, \dots, f_k) = \{x \in X \mid f_1(x) = \dots = f_k(x)\}. \]
This generalizes the coincidence set $\Coin(f_1,f_2) = \{ x \in X \mid f_1(x) = f_2(x)\}$ for two mappings.

Nielsen coincidence theory, see \cite{gonc05}, estimates the number of coincidence points of a pair of maps in a homotopy invariant way. Most of the techniques are a generalization of ideas from fixed point theory, see \cite{jian83}. In coincidence theory, one defines the \emph{Nielsen number} $N(f_1,f_2)$ of a pair of maps, which is a lower bound for the \emph{minimal coincidence number} $\MC(f_1,f_2)$:
\[ N(f_1,f_2) \le \MC(f_1,f_2) = \min \{ \#\Coin(f_1',f_2') \mid f_i'\htp f_i \}. \]
The above quantities are in fact equal when $X$ and $Y$ are compact
$n$-manifolds of the same dimension $n\neq 2$.
In this paper we extend this theory to equalizer sets. 

The typical
setting for Nielsen coincidence theory is for maps $X \to Y$ of
compact manifolds of the same dimension. For maps $f_1,f_2:X \to Y$ in
this setting, transversality arguments show that we can change the
maps by homotopy so that $\Coin(f_1,f_2)$ is a set of finitely many
points. At each of these points we define a \emph{coincidence index} which is then used to define the Nielsen number. In the case of differentiable manifolds we can define the index in terms of the determinant of the derivative maps at each coincidence point (see \cite{stae07a} for this approach).

When the dimension of $X$ is greater than that of $Y$, the typical approach to the coincidence index breaks down. In this case the derivative maps cannot be linear isomorphisms, and so their determinants cannot be used. A modified approach based on determinants is given by Jezierski in \cite{jezi01} which applies for maps into tori, but this is a fairly restrictive setting.

In this paper we will show that the typical approach, expressible in terms of determinants, does indeed succeed in the positive codimension setting when we admit more mappings to our theory, i.e.\ when we move from coincidence theory to equalizer theory. In this sense equalizer theory would seem to be the most natural and straightforward Nielsen-type theory in positive codimensions. Compared to the various approaches to positive codimension Nielsen coincidence theory (many are surveyed in \cite{gonc05}), our equalizer theory is substantially simpler and much more closely resembles classical Nielsen fixed point and coincidence theory. 

Nielsen equalizer theory will require a specific codimensional setting. Attempting a homotopy-invariant study of equalizers in codimension zero immediately gives:
\begin{thm}
If $X$ and $Y$ are compact manifolds of the same dimension, and $f_1, \dots, f_k: X \to Y$ are maps with $k>2$, then these maps can be changed by homotopy so that the equalizer set is empty.
\end{thm}
\begin{proof}
Well-known transversality arguments show that we can change $f_2$ by a homotopy to $f_2'$ so that $\Coin(f_1,f_2')$ is a finite set of points. Similarly we obtain $f_3' \htp f_3$ such that $\Coin(f_1,f_3')$ is a finite set of points. These homotopies can be arranged so that $\Coin(f_1,f_2')$ and $\Coin(f_1,f_3')$ are disjoint. Thus 
\[ \Eq(f_1, f'_2, f'_3, f_4, \dots, f_k) \subset \Coin(f_1,f_2') \cap \Coin(f_1,f_3') = \emptyset. \qedhere \]
\end{proof}

Thus there is no interesting theory for counting the minimal number of equalizer points between compact manifolds of the same dimension, since this number is always zero. In this sense, the equalizer equation $f_1(x) = \dots = f_k(x)$ is
``overdetermined'' when the dimensions of the domain and codomain are equal. In order to obtain an interesting theory we must increase the dimension of the domain space. In particular, for equalizers of $k$ maps, we will require $X$ and $Y$ to be of dimensions $(k-1)n$ and $n$, respectively, for any $n$. Consider the following example:
\begin{exl}\label{picexl}
We will examine the equalizer set of three maps $f,g,h:T^2 \to S^1$ from the 2-dimensional torus to the circle. Viewing the torus as the quotient of $\R^2$ by the integer lattice, and $S^1$ as the quotient of $\R$ by the integers, we will specify our maps by integer matrices of size $1 \times 2$. Let the maps be given by matrices:
\newcommand{\spa}{\quad}
\[ A_f = (3\spa 1),\quad A_g = (0 \spa 2),\quad A_h = (-1\spa -1). \]
Let $C_{fg} = \Coin(f,g)$, with $C_{fh}$ and $C_{gh}$ defined similarly, and we have \[ \Eq(f,g,h) = C_{fg} \cap C_{gh} \cap C_{fh}.\]
(Actually the equalizer set is the intersection of any two of these coincidence sets.) 

It is straightforward to compute these sets. For example, $C_{fg}$ is
the set of points $(x,y)$ with $3x+y = 2y \mod \Z^2$, which is to say $y=3x \mod
\Z^2$. Similarly computing the sets $C_{fh}$ and $C_{gh}$ produces the
picture in Figure \ref{exlfig}, where the torus is drawn as
$[0,1]\times [0,1]$ with opposite sides identified. We see in the
picture that $\Eq(f,g,h)$ consists of 10 points (the nine points where
the lines visibly intersect, plus the intersection at the identified
corners of the diagram). 

\begin{figure}
\[ 
\newcommand{\xwidth}{40}
\newcommand{\axline}{-}
\begin{xy}
(0,0)="ll";
(0,\xwidth) = "ul";
(\xwidth,0) = "lr";
(\xwidth,\xwidth) = "ur";
{\ar@{-} "ll";"ul"};
{\ar@{-} "ul";"ur"};
{\ar@{-} "ur";"lr"};
{\ar@{-} "lr";"ll"};
"ll"; "lr"; **\dir{} ?(.333) = "ltt"; ?(.666) = "lt";
"ul"; "ur"; **\dir{} ?(.333) = "utt"; ?(.666) = "ut";
{\ar@{-} "ll";"ut"};
{\ar@{-} "lt";"utt"};
{\ar@{-} "ltt";"ur"};
"ll"; "ul"; **\dir{} ?(.333)*{} = "ltt"; ?(.666)*{} = "lt";
"lr"; "ur"; **\dir{} ?(.333)*{} = "rtt"; ?(.666)*{} = "rt";
{\ar@{.} "ul";"rtt"};
{\ar@{.} "ltt";"rt"};
{\ar@{.} "lt";"lr"};
"ll";"ul"; **\dir{} ?(.5) = "lh";
"lr";"ur"; **\dir{} ?(.5) = "rh";
"ul";"ur"; **\dir{} ?(.25) = "ufff"; ?(.5)="uff"; ?(.75)="uf";
"ll";"lr"; **\dir{} ?(.25) = "lfff"; ?(.5)="lff"; ?(.75)="lf";
{\ar@{--} "lh";"lf"};
{\ar@{--} "ul";"lff"};
{\ar@{--} "uf";"lfff"};
{\ar@{--} "uff";"lr"};
{\ar@{--} "ufff";"rh"};
{\ar@{-} (50,30)*{C_{fg}:\quad}; (65,30)}
{\ar@{.} (50,20)*{C_{gh}:\quad}; (65,20)}
{\ar@{--} (50,10)*{C_{fh}:\quad}; (65,10)}
\end{xy}
\quad 
\]
\caption{Coincidence sets and equalizer points for Example \ref{picexl}. \label{exlfig}}
\end{figure}
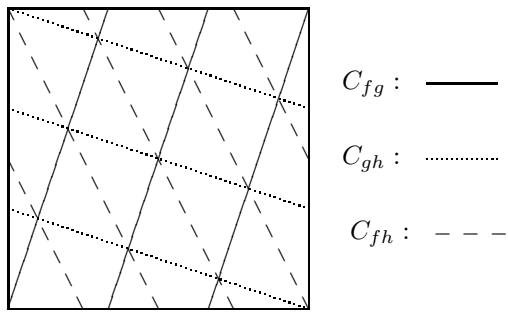

In this paper we will define the Nielsen number $N(f,g,h)$ which is a
lower bound for the minimum number of equalizer points when the maps
are changed by homotopy. In Theorem \ref{torusthm} we give a simple
formula for computing this quantity on tori, which in this example gives
\[ N(f,g,h) = 
\left| \begin{pmatrix} 0&2\\-1&-1 \end{pmatrix} - \begin{pmatrix} 3&1\\3&1\end{pmatrix}\right| = 10. \]
Thus these maps cannot be changed by homotopy to have fewer than 10 equalizer points.
\end{exl}

The construction of the theory is facilitated by a fundamental correspondence between $\Eq(f_1, \dots, f_k)$ and the coincidence set of a pair of related maps. 
Let $F,G:X \to Y^{k-1}$ be given by 
\begin{equation}\label{FGdef}
F(x) = (f_1(x), \dots, f_1(x)), \quad G(x) = (f_2(x), \dots, f_k(x)). \tag{$\star$}
\end{equation}
Since $X$ and $Y$ are compact with dimensions $(k-1)n$ and $n$ respectively, the
above $F$ and $G$ are maps between compact manifolds of the
same dimension, and $\Coin(F,G) = \Eq(f_1, \dots, f_k)$. This
correspondence is well-behaved under homotopy, since changing the maps
$f_i$ by homotopies corresponds in a natural way to a change of $F$
and $G$ by homotopies. As we shall see, the homotopy-invariant
behavior of $\Eq(f_1, \dots, f_k)$ is the same as that of
$\Coin(F,G)$, and we may define Nielsen-type invariants for the
equalizer set in terms of the same invariants from the coincidence
theory of $(F,G)$.  

In Section \ref{classessection} we define the Reidemeister and equalizer classes
which form the building blocks for our theory. In Section
\ref{nielsensection} we define the Nielsen number and in Section
\ref{compsection} we give some computational results for maps into
Jiang spaces and maps of tori. In Section \ref{positivecodimsection}
we give an application to Nielsen coincidence theory in positive
codimensions, giving a full computation of the ``geometric Nielsen
number'' on tori.

We would like to thank Robert F. Brown for helpful comments, and Philip Heath for bringing the reference \cite{jezi01} to our attention.

\section{Reidemeister and equalizer classes}\label{classessection}
Let $X$ and $Y$ be spaces with universal covering spaces (connected, locally path-connected, and semilocally simply connected), and let $\lift X$ and $\lift Y$ be the universal covering spaces with projection maps $p_X:\lift X \to X$ and $p_Y:\lift Y \to Y$. For maps $f_1, \dots, f_k: X \to Y$,
we wish to construct a Reidemeister-type theory for the equalizer points $\Eq(f_1, \dots, f_k)$, so that each point has an algebraic Reidemeister class, and two equalizer points can be combined by homotopy only when their classes are equal.

Our basic result is a generalization of a well-known result from coincidence
theory which is stated in part (without proof) as Lemma 2.3 of \cite{dj93}. For the
sake of completeness we give a full proof. The proof is similar to
that of Theorem 1.5 in \cite{jian83}, which is the corresponding statement in
fixed point theory. Throughout, elements of the fundamental group are
viewed as deck transformations on the universal covering space.

\begin{thm}\label{reidthm}
Let $f_1,\dots,f_k:X \to Y$ be maps with lifts $\lift f_i:\lift X \to \lift Y$ and induced homomorphisms $\phi_i:\pi_1(X) \to \pi_1(Y)$.
\begin{enumerate}
\item We have
\[ \Eq(f_1, \dots, f_k) = \bigcup_{\alpha_2, \dots, \alpha_k \in \pi_1(Y)} p_X \Eq(\lift f_1, \alpha_2\lift f_2, \dots, \alpha_k\lift f_k). \]
\item \label{disjointorequal}
For $\alpha_i, \beta_i \in \pi_1(X)$, the sets
\[ p_X \Eq(\lift f_1, \alpha_2\lift f_2, \dots, \alpha_k\lift f_k) \text{ and } p_X \Eq(\lift f_1, \beta_2\lift f_2, \dots, \beta_k\lift f_k) \]
are disjoint or equal. 
\item \label{coinreidclass} The above sets are equal if and only if there is some $z \in \pi_1(X)$ with 
\[ \beta_i = \phi_1(z) \alpha_i \phi_i(z)^{-1} \]
for all $i$.
\end{enumerate}
\end{thm}
\begin{proof}
For the first statement, take some $x \in \Eq(f_1,\dots,f_k)$ and some $\lift x \in p_X^{-1}(x)$. We have $p_Y(\lift f_i(\lift x)) = f_i(x) = f_1(x)$ for all $i$, and thus the values $\lift f_i(\lift x)$ all differ by deck transformations. That is, there are $\alpha_i \in \pi_1(Y)$ with
\[ \lift f_1(\lift x) = \alpha_2 \lift f_2(\lift x) = \dots = \alpha_k \lift f_k(\lift x), \]
which is to say that $\lift x \in \Eq(\lift f_1, \alpha_2\lift f_2, \dots, \alpha_k \lift f_k)$, and so $x \in p_X \Eq(\lift f_1, \alpha_2\lift f_2, \dots, \alpha_k \lift f_k)$ as desired.

Now we prove statement \ref{coinreidclass}. First, let us assume that $p_X
\Eq(\lift f_1, \alpha_2\lift f_2, \dots, \alpha_k\lift f_k) = p_X
\Eq(\lift f_1, \beta_2\lift f_2, \dots, \beta_k\lift f_k)$. This means
that for any point $\lift
x \in \Eq(\lift  f_1, \alpha_2\lift f_2, \dots, \alpha_k\lift f_k)$,
there is some deck transformation $z\in \pi_1(X)$ with $z \lift x \in \Eq(\lift f_1,
\beta_2\lift f_2, \dots, \beta_k\lift f_k)$. Then we have
\[
\beta_i \lift f_i(z\lift x) = \lift f_1(z\lift x) = \phi_1(z) \lift
f_1(\lift x) = \phi_1(z) \alpha_i \lift f_i(\lift x) = \phi_1(z)
\alpha_i \phi_i(z)^{-1} \lift f_i(z\lift x)
\]
Since the two lifts $\beta_i \lift f_i$ and
$\phi_1(z)\alpha_i\phi_i(z)^{-1}\lift f_i$ agree at a point, they are
the same lift, and thus 
\[ \beta_i = \phi_1(z)\alpha_i\phi_i(z)^{-1} \]
as desired.

For the converse in statement \ref{coinreidclass}, assume that
$\beta_i = \phi_1(z)\alpha_i\phi_i(z)^{-1}$ for all $i$, and take $x
\in p_X \Eq(\lift f_1, \alpha_2\lift f_2, \dots, \alpha_k\lift
f_k)$. Then we have $\phi_1(z)\alpha_i=\beta_i\phi_i(z) $ for all
$i$, and so
\[ \lift f_1(z\lift x) = \phi_1(z)\lift f_1(\lift x) = \phi_1(z)
\alpha_i \lift f_i(\lift x) = \beta_i \phi_i(z) \lift f_i(\lift x) =
\beta_i \lift f_i(z\lift x). \]
Thus $z \lift x = \Eq(\lift f_1, \beta_2\lift f_2, \dots,
\beta_k\lift f_k)$, and so $x \in p_X \Eq(\lift f_1, \beta_2\lift f_2, \dots,
\beta_k\lift f_k)$, and we have shown
\[ p_X \Eq(\lift f_1, \alpha_2\lift f_2, \dots, \alpha_k\lift
f_k) \subset p_X \Eq(\lift f_1, \beta_2\lift f_2, \dots,
\beta_k\lift f_k). \]
A symmetric argument shows the converse inclusion, and so the above
sets are equal.

For statement \ref{disjointorequal}, it suffices to show that if
there is a point 
\[ x \in p_X\Eq(\lift f_1, \alpha_2\lift f_2, \dots, \alpha_k\lift
f_k) \cap p_X\Eq(\lift f_1, \beta_2\lift f_2, \dots, \beta_k\lift
f_k), \]
then the two sets of the above intersection are equal. For such a
point $x$, there are $\lift x_0,\lift x_1 \in p_X^{-1}(x)$ with 
\[ \lift x_0 \in \Eq(\lift f_1, \alpha_2\lift f_2, \dots, \alpha_k\lift
f_k), \quad \lift x_1 \in \Eq(\lift f_1, \beta_2\lift f_2, \dots, \beta_k\lift
f_k). \]
Let $z\in \pi_1(X)$ with $z\lift x_0 = \lift x_1$. Then we have 
\begin{align*}
\beta_i\lift f_i(z\lift x_0) &= \beta_i \lift f_i(\lift x_1) = \lift
f_1(\lift x_1) = \lift f_1(z\lift x_0) = \phi_1(z)\lift f_1(\lift x_0)
\\
&= \phi_1(z)\alpha_i\lift f_i(\lift x_0) = \phi_1(z)\alpha_i
\phi_i(z)^{-1} \lift f_i(z\lift x_0).
\end{align*}
The above equality shows two lifts of $f_i$ agreeing at the point
$z\lift x_0$, and
so we have $\beta_i = \phi_1(z)\alpha_i \phi_i(z)^{-1}$, which by
statement \ref{coinreidclass} implies that 
\[ p_X\Eq(\lift f_1, \alpha_2\lift f_2, \dots, \alpha_k\lift
f_k) = p_X\Eq(\lift f_1, \beta_2\lift f_2, \dots, \beta_k\lift
f_k)\]
as desired.
\end{proof}

Let $\Reid(\phi_1, \dots, \phi_k) = \pi_1(Y)^{k-1} / \sim$ be the quotient of $\pi_1(Y)^{k-1}$ by the following relation, inspired by statement \ref{coinreidclass} above:
\[ (\alpha_2, \dots, \alpha_k) \sim (\beta_2, \dots, \beta_k) \]
if and only if there is some $z\in \pi_1(X)$ with 
\[ \beta_i = \phi_1(z) \alpha_i \phi_i(z)^{-1} \]
for all $i\in \{2,\dots, k\}$. We call $\Reid(\phi_1,\dots,\phi_k)$ the set of \emph{Reidemeister classes} for $\phi_1,\dots,\phi_k$.

Then the theorem above gives the following disjoint union
\begin{equation*}\label{reiddj}
\Eq(f_1,\dots, f_k) = \bigsqcup_{(\alpha_i) \in \Reid(\phi_1, \dots, \phi_k)} p_X \Eq(\lift f_1, \alpha_2 \lift f_2, \dots, \alpha_k \lift f_k).
\end{equation*}

The above union partitions the equalizer set into \emph{Nielsen
  equalizer classes} (or simply \emph{equalizer classes}). That is, $C
\subset \Eq(f_1,\dots,f_k)$ is an equalizer class if and only if there
are $\alpha_i$ with $C = p_X \Eq(\lift f_1, \alpha_2\lift f_2, \dots,
\alpha_k \lift f_k)$. Note that an equalizer class can be empty. The
equalizer classes are related to the coincidence classes of the pair
$(F,G)$ from equation \eqref{FGdef} in the following way:
\begin{thm}\label{coineqclasses}
A subset $C \subset \Eq(f_1, \dots, f_k)$ is an equalizer class if and
only if $C$ is a coincidence class when regarded as a subset of
$\Coin(F,G)$. That is, $C$ is an equalizer class if and only if there
is a deck transformation $A \in \pi_1(Y^{k-1})$ with $C = p_X
\Coin(\lift F, A \lift G)$ for some lifts $\lift F$ and $\lift G$ of
$F$ and $G$. 
\end{thm}
\begin{proof}
First we assume that $C$ is an equalizer class, and so we have lifts $\lift f_i$ of $f_i$ and $\alpha_i \in \pi_1(Y)$ with $C = p_X \Eq(\lift f_1, \alpha_2\lift f_2, \dots, \alpha_k\lift f_k)$. 
Let $\lift F$ and $\lift G$ be given by
\[ \lift F(\lift x) = (\lift f_1(\lift x), \dots \lift f_1(\lift x)), \quad \lift G(\lift x) = (\lift f_2(\lift x), \dots, \lift f_k(\lift x)), \]
and let $A:\lift Y^{k-1} \to \lift Y^{k-1}$ be
\[ A(\lift y_2, \dots \lift y_k) = (\alpha_2\lift y_2, \dots, \alpha_k\lift y_k). \]
Then we have
\[ \Eq(\lift f_1, \alpha_1\lift f_2, \dots, \alpha_k\lift f_k) = \Coin(\lift F, A \lift G), \]
and so $C = p_X\Coin(\lift F,A\lift G)$ as desired.

Now for the converse we assume that $C$ is a coincidence class of $(F,G)$, which means there are lifts $\lift F$ and $\lift G$ of $F$ and $G$ with a deck transformation $A \in \pi_1(Y^{k-1})$ such that $C = p_X \Coin(\lift F, A\lift G)$. Since $\lift F$ and $\lift G$ are lifts of $F$ and $G$, we can write
\[ \lift F(\lift x)  = (\lift f_1^2(\lift x), \dots, \lift f_1^k(\lift x)), \quad \lift G(\lift x) = (\lift f_2(\lift x), \dots, \lift f_k(\lift x)) \]
where each $\lift f_1^i$ is a lift of $f_1$, and $\lift f_j$ is a lift of $f_j$ for $j\ge 2$. Similarly we may factor $A$ as $A = \alpha_1 \times \dots \times \alpha_k$ for $\alpha_i \in \pi_1(Y)$. 

Each of the $\lift f_1^i$ may be different, but there is a single lift $\lift f_1$ of $f_1$ with deck transformations $\beta_i$ such that $\beta_i\lift f_1 = \lift f_1^i$. Then we have
\begin{align*}
\Coin(\lift F,A\lift G) &= \Coin((\beta_2 \lift f_1, \dots, \beta_k\lift f_1), (\alpha_2\lift f_2, \dots, \alpha_k\lift f_k)) \\
&= \Coin((\lift f_1, \dots, \lift f_1), (\beta_2^{-1}\alpha_2\lift f_2, \dots, \beta_k^{-1}\alpha_k\lift f_k)) \\
&= \Eq(\lift f_1, \beta_2^{-1}\alpha_2 \lift f_2, \dots, \beta_k^{-1}\alpha_k \lift f_k)
\end{align*}
and so $C = p_X \Coin(\lift F, A \lift G)$ is an equalizer class.
\end{proof}

The equalizer classes can be described nicely in terms of paths in $X$ and their images under the $f_i$:
\begin{thm}
Two points $x, x' \in \Eq(f_1,\dots,f_k)$ are in the same equalizer class if and only if there is some path $\gamma:[0,1] \to X$ from $x$ to $x'$ such that $f_1(\gamma)$ and $f_i(\gamma)$ are homotopic as paths with fixed endpoints for all $i$.
\end{thm}
\begin{proof}
Our points $x,x'$ are in the same equalizer class if and only if they are in the same coincidence class of the pair $(F,G)$. A standard result in coincidence theory shows that this is equivalent to the existence of a path $\gamma$ in $X$ from $x$ to $x'$ with $F(\gamma) \htp G(\gamma)$. This is equivalent to 
\[ (f_1,\dots,f_1)(\gamma) \htp (f_2,\dots,f_k)(\gamma), \]
which is equivalent to $f_1(\gamma) \htp f_i(\gamma)$ for each $i$.
\end{proof}

\section{The equalizer index and the Nielsen number}\label{nielsensection}
Let $\Eq(f_1, \dots, f_k, U) = \Eq(f_1, \dots, f_k) \cap U$, and let $\Coin(f,g,U) = \Coin(f,g) \cap U$. 

Our index for equalizer sets will be defined in terms of the coincidence index $i$. We first review some properties of the coincidence index. Let $f,g:M \to N$ be maps between compact orientable manifolds of the same dimension.
The coincidence index $i(f,g,U)$ is an integer valued function defined for open sets $U$ with $\Coin(f,g,U)$ compact. It satisfies the following properties:

\begin{itemize}
\item \emph{Homotopy:} Let $f'\htp f$ and $g' \htp g$, by homotopies $F_t$ and $G_t$, such that the set 
\[ \{(x,t) \mid x \in \Coin(F_t,G_t,U) \} \subset M \times [0,1] \]
is compact (such a pair of homotopies is called \emph{admissible}). Then $i(f,g,U) = i(f',g',U)$.
\item \emph{Additivity:} If $U_1 \cap U_2 = \emptyset$ and $\Coin(f,g,U) \subset U_1 \cup U_2$, then 
\[ i(f,g,U) = i(f,g,U_1) + i(f,g,U_2). \]
\item \emph{Solution:} If $i(f,g,U) \neq 0$, then $\Coin(f,g,U)$ is not empty.
\end{itemize}

We wish to define a similar index in the equalizer setting. Let $X$ and $Y$ be compact orientable manifolds of dimensions $(k-1)n$ and $n$, respectively, with maps $f_1, \dots, f_k:X \to Y$. 
We call $(f_1, \dots, f_k, U)$ \emph{admissible} when $\Eq(f_1, \dots, f_k, U)$ is compact.

Let $F,G:X \to Y^{k-1}$ be the maps as in \eqref{FGdef}. These are 
maps between compact orientable manifolds of the same dimension. When
$(f_1,\dots,f_k,U)$ is admissible, then $\Coin(F,G,U) = \Eq(f_1,
\dots, f_k,U)$ is compact, and thus the coincidence index $i(F,G,U)$
is defined. We define the equalizer index $\ind(f_1,\dots, f_k,U)$ to
be $i(F,G,U)$. 

This equalizer index satisfies the appropriate homotopy, additivity, and solution properties. If $(f_1,\dots,f_k,U)$ and $(f'_1,\dots,f'_k,U)$ are admissable and $f_i \htp f'_i$ with homotopy $H^i$, we say that $(H^i)$ is an \emph{admissible homotopy} of $(f_1, \dots, f_k, U)$ to $(f_1',\dots,f_k', U)$ when the set 
\[ \{ (x,t) \mid x \in \Eq(H^1_t, \dots, H^k_t, U) \} \subset X \times I  \]
is compact.

\begin{thm}\label{indthm}
Let $f_1,\dots,f_k:X \to Y$ be maps of compact orientable manifolds of dimensions $(k-1)n$ and $n$ respectively, and let $U\subset X$ be open with $(f_1,\dots,f_k,U)$ admissable.

Then the equalizer index $\ind(f_1,\dots,f_k,U)$ is defined and satisfies the following properties:
\begin{itemize} 
\item \emph{Homotopy:} If $(f_1, \dots, k_k,U)$ is admissibly homotopic to $(f'_1,\dots,f'_k,U)$, then $\ind(f_1,\dots,f_k,U) = \ind(f'_1,\dots,f'_k,U)$.
\item \emph{Additivity:} If $U_1 \cap U_2 = \emptyset$ and $\Eq(f_1,\dots,f_k,U) \subset U_1 \cup U_2$, then 
\[ \ind(f_1,\dots,f_k,U) = \ind(f_1,\dots,f_k,U_1) + \ind(f_1,\dots,f_k,U_2). \]
\item \emph{Solution:} If $\ind(f_1,\dots,f_k,U) \neq 0$, then $\Eq(f_1,\dots,f_k,U)$ is not empty.
\end{itemize}
\end{thm}
\begin{proof}
The proofs of these properties all follow from the same properties of
the coincidence index of the pair $F,G$ as in \eqref{FGdef}.
\end{proof}

For an equalizer class $C$, we define the \emph{index} of $C$, written
$\ind(f_1,\dots,f_k,C)$, as $\ind(F,G,U)$, where $U$ is an open set
with $\Coin(F,G,U) = C$ (such an open set will always exist because
coincidence classes are closed and $X$ is compact).

At this point we take a slight diversion to give a note on the computation of the index of differentiable maps in terms of their derivatives.
When $X$ and $Y$ are differentiable manifolds and each of $f_i$ is differentiable, the maps $F$ and $G$ will also be differentiable, and the derivative maps $DF_x, DG_x: \R^{(k-1)n} \to \R^{(k-1)n}$ are defined at each point $x \in X$. 

Let $x \in \Eq(f_1,\dots,f_k)$ be an equalizer point. We say that $x$ is \emph{nondegenerate} when $DG_x - DF_x$ is nonsingular. In this case there is a neighborhood $U$ around $x$ containing no other coincidence points of $F$ and $G$, and thus no other equalizer points, and the index can be computed by the well-known formula from coincidence theory:
\[ \ind(f_1,\dots,f_k,U) = i(F,G, U) = \sign \det(DG_x - DF_x). \]
The definitions of $F$ and $G$ give the following formula in terms of the $f_i$.

\begin{thm}
Let $f_1,\dots,f_k:X \to Y$ be maps of compact orientable manifolds of dimensions $(k-1)n$ and $n$ respectively, and let $x \in \Eq(f_1,\dots,f_k)$ be nondegenerate. 

Then there is a neighborhood $U$ of $x$ with $\Eq(f_1,\dots,f_k,U) = \{x\}$ such that
\[ \ind(f_1,\dots,f_k,U) = \sign \det\begin{pmatrix} df_{2} - df_{1} \\ \vdots \\ df_{k} - df_{1} \end{pmatrix} \]
where all derivatives are taken at the point $x$ (each row in the above is an $n \times (k-1)n$ block matrix, so that the whole matrix has size $(k-1)n\times (k-1)n$).
\end{thm}

Now we discuss the index theory for the nonorientable case. For the
coincidence theory of maps $f,g:M \to N$ of compact (perhaps
nonorientable) manifolds of the same dimension, an integer-valued coincidence index
cannot in general be defined. There is a related \emph{semi-index}
(see \cite{dj93}) which plays a similar role.

The semi-index, which we denote $|i|$, is defined not for arbitrary open sets, but only for coincidence classes, and satisfies properties similar to those of the coincidence index. Let $C \subset \Coin(f,g)$ be a coincidence class with $C= p\Coin(\lift f, \alpha \lift g)$. Then if $f\htp f'$ and $g \htp g'$, these homotopies will lift, producing maps $\lift f'\htp \lift f$ and $\lift g'\htp \lift g$ which are lifts of $f'$ and $g'$ respectively. Thus $D=p\Coin(\lift f',\alpha \lift g')$ is a coincidence class of $(f',g')$, and we say that $D$ is ``related to $C$'' with respect to the pair of homotopies.

If $f,g:M\to N$ are maps of compact manifolds of the same dimension and $C$ is a (possibly empty) coincidence class, then $|i|(f,g,C)$ is defined and satisfies:
\begin{itemize}
\item \emph{Homotopy:} If $f' \htp f$ and $g'\htp g$, and $D$ is the coincidence class of $(f',g')$ which is related to $C$ with respect to these homotopies, then $|i|(f,g,C) = |i|(f',g',D)$.
\item \emph{Solution:} If $|i|(f,g,C) \neq 0$, then $C$ is not empty.
\item \emph{Naturality:} If $M$ and $N$ are orientable, then $|i|(f,g,C) = |i(f,g,C)|$, the absolute value of the usual coincidence index.
\end{itemize}

In the setting of equalizer theory for maps $f_1,\dots,f_k:X \to Y$ of compact (possibly nonorientable) manifolds with an equalizer class $C$, we define the \emph{equalizer semi-index} as in the orientable case: let $(F,G)$ be as in \eqref{FGdef}, and we define $|\ind|(f_1,\dots,f_k,C) = |i|(F,G,C)$. Given homotopies $f'_i\htp f_i$, the ``relation'' between equalizer classes of $(f_1,\dots,f_k)$ and $(f'_1,\dots,f'_k)$ is defined exactly as in coincidence theory.

The following has routine proofs similar to those for Theorem \ref{indthm}.
\begin{thm}\label{semiind}
Let $f_1,\dots,f_k:X \to Y$ be maps of compact (possibly nonorientable) manifolds of dimensions $(k-1)n$ and $n$ respectively, and let $C \subset \Eq(f_1,\dots,f_k,U)$ be an equalizer class.

Then the equalizer semi-index $|\ind|(f_1,\dots,f_k,C)$ is defined and satisfies the following properties:
\begin{itemize} 
\item \emph{Homotopy:} If $f_i$ is homotopic to $f'_i$ for each $i$ and $D$ is the equalizer class of $(f'_1, \dots, f'_k)$ which is related to $C$, then \[ |\ind|(f_1,\dots,f_k,C) = |\ind|(f'_1,\dots,f'_k,D). \]
\item \emph{Solution:} If $|\ind|(f_1,\dots,f_k,C) \neq 0$, then $C$ is not empty.
\item \emph{Naturality:} If $X$ and $Y$ are orientable, then 
\[ |\ind|(f_1,\dots,f_k,C) = |\ind(f_1,\dots,f_k,C)|, \]
the absolute value of the usual equalizer index.
\end{itemize}
\end{thm}

An equalizer class is called \emph{essential} if its index (or semi-index in the nonorientable case) is nonzero.

\begin{defn} 
The \emph{Nielsen [equalizer] number} $N(f_1, \dots, f_k)$ is defined to be the number of essential equalizer classes of $(f_1, \dots, f_k)$.
\end{defn}

From Theorem \ref{coineqclasses} and the definition of the index of a class, we see that $N(f_1,\dots,f_k)$ is equal to the Nielsen coincidence number of the pair $(F,G)$. 
Since the Nielsen equalizer number is so closely related to a coincidence number, we can obtain a Wecken-type theorem for the minimal number of equalizer points. Let $\ME(f_1, \dots, f_k)$ be the \emph{minimal number of equalizer points}, defined as
\[ \ME(f_1, \dots, f_k) = \min \{ \#\Eq(f'_1,\dots,f'_k) \mid f'_i \htp f_i \}. \]

By the solution properties of the index and semi-index, every essential equalizer class must contain an equalizer point, and so 
\[ N(f_1,\dots,f_k) \le \ME(f_1,\dots,f_k). \]
These two quantities are in fact equal in most cases, as the following theorem shows.

\begin{thm}\label{wecken}
Let $f_1,\dots,f_k:X \to Y$ be maps of compact manifolds of dimensions $(k-1)n$ and $n$ respectively. If $(k-1)n \neq 2$, then 
\[ \ME(f_1,\dots,f_k) = N(f_1,\dots,f_k). \]

In the case of ``proper'' equalizer theory (when $k>2$), the result holds for all $k$ and $n$ except $(k,n) =(3,1)$, which is to say equalizer theory of three maps from a compact surface to the circle.
\end{thm}
\begin{proof}
The second statement is simply a consequence of $k,n$ being natural numbers with $(k-1)n \neq 2$, so we focus on the first statement.

Let $(F,G)$ be defined as in \eqref{FGdef}, and we have $N(f_1,\dots,f_k) = N(F,G)$. The maps $F,G$ are maps between compact manifolds of dimension $(k-1)n$. By our hypothesis this dimension is not 2, and so the Wecken theorem for coincidences (see \cite{gonc05}) gives maps $F'\htp F$ and $G'\htp G$ with $\#\Coin(F',G') = N(F,G)$.  A result of Brooks in \cite{broo71} shows that in fact there is a single map $G'' \htp G$ with $\Coin(F',G') = \Coin(F,G'')$, and thus $\#\Coin(F,G'') = N(F,G)$.

Our map $G''$ is a map of $X \to Y^{k-1}$, so it can be written as $G''(x) = (g_2(x), \dots, g_k(x))$ with $g_i \htp f_i$. Now we have
\[ \#\Eq(f_1, g_2, \dots, g_k) = \#\Coin(F, G'') = N(F,G) = N(f_1,\dots,f_k), \]
and so $\ME(f_1,\dots,f_k) \le N(f_1,\dots,f_k)$ as desired.
\end{proof}

\section{Some computations}\label{compsection}
\subsection{Jiang spaces}
One setting in which the fixed point and coincidence Nielsen numbers
are easily calculated is for maps on \emph{Jiang spaces}. See
\cite{jian83} for the definition and basic results in fixed point
theory. The class of Jiang spaces includes
topological groups, generalized lens spaces and certain other homogeneous
spaces, and is closed under products.
The main result (see \cite{gonc05}) from coincidence theory concerning Jiang spaces is the following:
\begin{thm}\label{coinjiang}
If $f,g:M \to N$ are maps of compact orientable manifolds of the same dimensions and $N$ is a Jiang space, then every coincidence class has the same index.
\end{thm}

Our theorem concerning Jiang spaces is the following result, which is facilitated by the coincidence theory of the maps $(F,G)$ as in \eqref{FGdef}.
\begin{thm}\label{jiangthm}
If $f_1,\dots,f_k:X \to Y$ are maps of compact orientable manifolds of dimensions $(k-1)n$ and $n$ respectively and $Y$ is a Jiang space, then every equalizer class has the same index.
\end{thm}
\begin{proof}
Let $F,G:X \to Y^{k-1}$ be given as in \eqref{FGdef}:
\[ F(x) = (f_1(x), \dots f_1(x)), \quad G(x) = (f_2(x), \dots, f_k(x)).  \]

Since $Y$ is a Jiang space, then $Y^{k-1}$ is a Jiang space. Thus by Theorem \ref{coinjiang} all coincidence classes of $F,G$ will have the same coincidence index. But the equalizer classes of $f_1,\dots,f_k$ are the same as the coincidence classes of $F,G$, with the same indices, so all equalizer classes of $f_1,\dots,f_k$ will have the same equalizer index.
\end{proof}

Define the \emph{Reidemeister number} and \emph{Lefschetz number} as: $R(f_1,\dots,f_k) = \#\Reid(\phi_1,\dots, \phi_k)$ (this quantity may be infinite) and $L(f_1,\dots,f_k) = \ind(f_1,\dots,f_k,X)$. Then we obtain:
\begin{cor}
If $f_1,\dots,f_k:X \to Y$ are maps of compact orientable manifolds of dimensions $(k-1)n$ and $n$ respectively, and $Y$ is a Jiang space, then:
\begin{itemize}
\item If $L(f_1,\dots,f_k) = 0$ then $N(f_1,\dots,f_k) = 0$.
\item If $L(f_1,\dots,f_k) \neq 0$ then $N(f_1,\dots,f_k) = R(f_1,\dots,f_k)$.
\end{itemize}
\end{cor}
\begin{proof}
By the additivity property, $L(f_1,\dots,f_k)$ is the sum of the
indices of each equalizer class. By Theorem \ref{jiangthm}
all classes have the same index, thus
$L(f_1,\dots,f_k) = 0$ means that all classes are inessential and so
$N(f_1,\dots,f_k) = 0$. If the Lefschetz number is not zero 
then all classes are essential and so the
Nielsen number is simply the number of classes, which is the
Reidemeister number. 
\end{proof}

\subsection{Tori}
We can give a very specific formula for the Nielsen number of maps $f_1,\dots,f_k:T^{(k-1)n} \to T^n$ on tori. We will view $T^m$ as the quotient of $\R^m$ by the integer lattice, and consider maps which are induced by linear maps on $\R^{(k-1)n} \to \R^n$ taking $\Z^{(k-1)n}$ to $\Z^n$. We can think of such a map as an $n\times(k-1)n$ matrix with integer entries, which we call the ``induced matrix''. 

We now prove the formula which was used in the computation of Example
\ref{picexl}. The result generalizes the well-known formula for the
Nielsen coincidence number on tori which was proved in Lemma 7.3 of
\cite{jezi90}: if $f_1, f_2$ are given by square matrices $A_1$ and
$A_2$, then $N(f_1,f_2) = |\det(A_2-A_1)|$. 

\begin{thm}\label{torusthm}
If $f_1,\dots,f_k:T^{(k-1)n} \to T^n$ are maps on tori with induced matrices $A_i$, then 
\[ N(f_1,\dots,f_k) = \left| \det\begin{pmatrix} A_2 - A_1 \\ \vdots
    \\ A_k - A_1 \end{pmatrix} \right| \]
\end{thm}
\begin{proof}
Let $F,G:T^{(k-1)n} \to T^{(k-1)n}$ be as in \eqref{FGdef}. Then the induced matrices of $F$ and $G$ will be given by block matrices
\[ A_F = \begin{pmatrix}A_1\\\vdots\\A_1\end{pmatrix}, \quad A_G = \begin{pmatrix}A_2\\\vdots\\A_k\end{pmatrix}, \]
and so the formula for the Nielsen coincidence number on tori gives
\[ N(F,G) = \left|\det \begin{pmatrix} A_2 - A_1 \\ \vdots \\ A_k -
    A_1 \end{pmatrix} \right| \]
But $N(F,G) = N(f_1,\dots,f_k)$, and so the result is proved.
\end{proof}

We further note that since tori have the Wecken property for coincidence theory, we can drop the dimension assumption of Theorem \ref{wecken}.
\begin{thm}\label{toriwecken}
Let $f_1,\dots,f_k:T^{(k-1)n} \to T^n$ be maps of tori. Then 
\[ \ME(f_1,\dots,f_k) = N(f_1,\dots,f_k). \]
\end{thm}
\begin{proof}
Let $(F,G)$ be as in \eqref{FGdef}, and then since tori have the Wecken property there is a map $G'' \htp G$ with $\#\Coin(F,G'') = N(F,G)$. We finish the argument as in the last paragraph of the proof of Theorem \ref{wecken}.
\end{proof}

\section{Coincidence theory with positive codimension}\label{positivecodimsection}
We end with an application to coincidence theory with positive
codimension, which typically requires much more difficult
techniques than those of this paper. In this setting we consider maps
$f_1,f_2:X \to Y$ of 
compact manifolds of dimensions $m$ and $n$ with $m>n$ and try to
minimize by homotopies the quantity
$\# \pi_0(\Coin(f_1,f_2))$, the number of path components of $\Coin(f_1,f_2)$.

There is no coincidence index in the positive codimension setting, and
so the problem of judging essentiality of classes is more complicated. A
coincidence class $C \subset \Coin(f_1,f_2)$ is \emph{removable by
  homotopy} when there is some pair of homotopies $f_i \htp f_i'$ such
that $C$ is ``related'' (in the sense of Theorem \ref{semiind}) to the
empty class. When a class is not removable by homotopy, it is called
\emph{geometrically essential}. The number of geometrically essential
classes is called the \emph{geometric Nielsen number}, which we denote
$N_G(f_1,f_2)$.

Any two coincidence points which can be connected by a path of
coincidence points will be in the same coincidence class. Thus 
each class is a union of path components of
$\Coin(f_1,f_2)$, and so $N_G(f_1,f_2) \le
\# \pi_0(\Coin(f_1,f_2))$. Since $N_G(f_1,f_2)$ is homotopy invariant, in fact
it is a lower bound for the minimal number of path
components of the coincidence set when $f_1$ and $f_2$ are changed by
homotopies. 

We begin with a simple result which in some cases can demonstrate that
a coincidence class is geometrically essential.
\begin{thm}\label{geomess}
Let $f_1,\dots,f_k:X \to Y$ be maps of spaces of dimension $(k-1)n$ and $n$ respectively, and let $f_i,f_j$ be any two of these maps. Then each equalizer class of $(f_1,\dots,f_k)$ is a subset of some coincidence class of $(f_i,f_j)$, and any coincidence class containing an essential equalizer class is geometrically essential.
\end{thm}
\begin{proof}
To show that each equalizer class is a subset of a coincidence class, let $C$ be an equalizer class. Then there are lifts $\lift f_i$ and deck transformations $\alpha_i$ with 
\[C = p_X\Eq(\alpha_1 \lift f_1, \alpha_2\lift f_2, \dots, \alpha_k\lift f_k) \subset p_X \Coin(\alpha_i \lift f_i, \alpha_j \lift f_j), \]
and the right side above is a coincidence class.

Now let $D \subset \Coin(f_i,f_j)$ be a coincidence class containing some essential equalizer class $C \subset D$. If $D$ were removable by a homotopy as a coincidence class, then necessarily $C$ would be removable by a homotopy as an equalizer class, which is impossible since $C$ is essential. Thus $D$ is geometrically essential.
\end{proof}

We can state the above in terms of Nielsen numbers:
\begin{cor}\label{easyngn}
Let $f_1,f_2:X \to Y$ be maps of spaces of dimension $(k-1)n$ and $n$
respectively. If there are maps $f_3, \dots, f_k$ with
$N(f_1,\dots,f_k)\neq 0$, then $N_G(f_1,f_2) \neq 0$.
\end{cor}
\begin{proof}
If $N(f_1,\dots,f_k) \neq 0$ then there is
an essential equalizer class of $(f_1,\dots,f_k)$, which by Theorem
\ref{geomess} is contained in a geometrically essential coincidence
class of $(f_1,f_2)$. The existence of this coincidence class means
that $N_G(f_1,f_2) \neq 0$.
\end{proof}

Now we focus on tori, for which we can be much more specific about the
value of $N_G(f_1,f_2)$. As we will see, Corollary \ref{easyngn} is strong
enough to give a complete computation of $N_G(f_1,f_2)$ based on the matrices
which specify the maps, even in the case where the domain dimension is
not a multiple of the codomain dimension. 

Theorem \ref{coker} below, computing the value of $N_G(f_1,f_2)$ on tori, is proved by Jezierski in \cite{jezi01}. Most of the argument follows exactly the proof in the codimension zero case given in \cite{jezi90}. The key novel step in the positive codimension setting is the following lemma. 

\begin{lem}\label{fullrankng}
Let $f_1,f_2:T^{m} \to T^{n}$
be maps of tori with induced matrices $A_1,A_2$.
If $A_2-A_1$ has rank $n$, then $N_G(f_1,f_2) \neq 0$.
\end{lem}

Jezierski proves this by using the fact that, for $m>n$, we can consider $T^n$ as a subspace of $T^m$ and then note that the restrictions of $f_1, f_2$ will have a nonremovable coincidence class. Jezierski's approach effectively decreases the domain dimension in order to apply the classical  codimension zero theory. 

We take the opposite approach of \emph{increasing} the domain dimension by taking a product with circles and introducing additional maps $f_3, \dots, f_k$ which allow us to apply Corollary \ref{easyngn}. While Jezierski's approach is simpler for this particular argument, it relies strongly on the fact that there is a standard embedding of $T^n$ inside $T^m$. Since we do not use this fact, we hope that our strategy may be useful in other settings. 

\begin{proof}[Proof of Lemma \ref{fullrankng}]
First we consider the case where $m=(k-1)n$ for some $k$. In this case 
we can choose matrices $A_3, \dots, A_k$ so that 
\[ \det\begin{pmatrix} A_2-A_1\\\vdots\\A_k-A_1\end{pmatrix} \neq
0, \]
and so there are maps $f_3,\dots,f_k$ with $N(f_1,\dots,f_k) \neq
0$. By Corollary \ref{easyngn} this implies that $N_G(f_1,f_2) \neq 0$.

For general $m$, 
let $k>2$ be an integer with $(k-1)n \ge m$. Then define $g_1,
g_2:T^{(k-1)n} \to T^n$ as $g_i = f_i \circ \sigma$, where $\sigma:
T^{(k-1)n} \to T^m$ is the projection onto the first $m$ coordinates
(viewing the torus as a product of circles).
Let $B_i$ be the $(k-1)n \times n$ integer
matrix representing $g_i$. As a matrix, $B_i$ is simply 
$A_i$ with columns of zeros added, and so the rank of $A_2-A_1$ is
the same as that of $B_2 - B_1$. 

Our assumption that $A_2-A_1$ has rank $n$ means that $B_2 - B_1$ has
rank $n$, and so by our first case we have $N_G(g_1,
g_2) \neq 0$, and we have a geometrically
essential coincidence class $D \subset \Coin(g_1,g_2)$. Let $\lift
f_i$ be lifts of $f_i$, and let $\lift g_i = \lift f_i \circ \lift
\sigma$, where $\lift \sigma$ is the projection onto the first $m$
coordinates of $\lift T^{(k-1)n} = \R^{(k-1)n}$. Then $\lift g_i$ is a 
lift of $g_i$, and so there is some $\alpha \in \pi_1(T^n)$ with 
\[ D = p \Coin(\lift g_1, \alpha \lift g_2). \]

Let $x \in \sigma(D)$, so there is some $y$ with $\sigma(y) = x$ and a
lift $\lift y = p^{-1}(y)$ with $\lift g_1(\lift y) = \alpha \lift
g_2(\lift y)$, and thus $\lift f_1( \lift \sigma(\lift y)) = \alpha
\lift f_2(\lift \sigma(\lift y))$. Since $p(\lift \sigma(\lift y)) =
\sigma(y) = x$, we have $x \in p\Coin(\lift f_1, \alpha \lift
f_2)$. This set $C = p\Coin(\lift f_1, \alpha \lift
f_2)$ is a coincidence class of $(f_1,f_2)$, and we have shown that
$\sigma(D) \subset C$.

Recall that we are trying to show that $N_G(f_1,f_2) \neq 0$. For the
sake of a contradiction, assume that $N_G(f_1,f_2) = 0$, so that each
class (in particular the class $C$) is removable by homotopy. This
means there are maps $f'_i \htp f_i$ with lifts 
$\lift f'_i \htp \lift f_i$ such that 
\begin{equation}\label{removed}
p \Coin(\lift f'_1, \alpha \lift f'_2) = \emptyset. 
\end{equation}

Let $g'_i = f'_i \circ \sigma$ and $\lift g'_i = \lift f'_i \circ
\lift \sigma$. Then $g'_i \htp g_i$ and $\lift g'_i \htp \lift
g_i$. Since $D$ is geometrically essential, the related class
$p\Coin(\lift g'_1, \alpha \lift g'_2)$ must be nonempty.
Take some $y$ in this class and a point $\lift y \in p^{-1}(y)$ with $\lift
g'_1(\lift y) = \alpha \lift g'_2(\lift y)$. Then
we have $\lift f'_1(\lift \sigma(\lift y)) = \alpha \lift f'_2(\lift
\sigma(\lift y))$, and so
\[ \sigma (y) \in p \Coin(\lift f'_1, \alpha \lift f'_2), \]
which contradicts \eqref{removed}.
\end{proof}

The above provides the key step in the proof of the following complete computation of the geometric Nielsen number on tori. The argument in codimension zero given in \cite{jezi90} carries without modification in arbitrary codimension, except for this step. Jezierski presents the details, along with a different argument substituting for Lemma \ref{fullrankng} in \cite{jezi01}.

\begin{thm}\label{coker}
Let $f_1,f_2:T^{m} \to T^n$ be maps of tori with induced matrices
$A_1,A_2$. If $A_2-A_1$ has rank $n$, then 
\[ N_G(f_1,f_2) = \#\pi_0(\Coin(f_1,f_2)) = \#\coker(A_2-A_1), \]
where $\coker(A_2-A_1) = \Z^n / \im(A_2 - A_1)$, the cokernel of
$A_2-A_1$ when viewed as a homomorphism $\Z^m \to \Z^n$.
\end{thm}

As a brief illustration, we compute the geometric
Nielsen coincidence numbers for the maps $f,g,h:T^2 \to S^1$ from
Example \ref{picexl}.
\begin{exlctd}
Recall our maps were given by matrices:
\[ A_f = (3\quad 1),\quad A_g = (0 \quad 2),\quad A_h = (-1\quad -1). \]
For each pair of matrices the rank assumption of Theorem \ref{coker}
holds.

It is straightforward to compute the required cokernels. We have $A_g - A_f = (3
\quad -1)$, and so $\im(A_g - A_f) = \Z$, since $\gcd(3,-1) = 1$. Thus
the cokernel is trivial and so $N_G(f,g) = 1$. A similar computation
shows that $N_G(g,h) = 1$. For $(f,h)$, we have $A_h - A_f = (2 \quad
-2)$, and so $\im(A_h - A_f) = 2\Z$. Thus the cokernel is $\Z/2\Z$,
and so $N_G(f,h) = 2$.

By Theorem \ref{coker} these Nielsen numbers should agree with the
number of path components of the coincidence sets. Counting components
in Figure \ref{exlfig} indeed gives
$\#\pi_0(\Coin(f,g))=\#\pi_0(\Coin(g,h))=1$ and $\#\pi_0(\Coin(f,h))=2$. 
\end{exlctd}

\section*{Addendum}
The author was fortunate to attend the International Conference on Nielsen Theory and Related Topics at Capitol Normal University in Beijing in Summer 2011 after this article was initially published. There he learned from Peter Wong that a Nielsen type equalizer theory in the category of smooth orientable manifolds was obtained as a special case of a more general theory by Dobre\'nko and Kucharsky in \cite{dk90}. 

For smooth orientable manifolds $M$ and $N$ and a submanifold $B \subset N$ with $\dim X + \dim B = \dim Y$, Dobre\'nko and Kucharsky define a general Nielsen theory for counting the cardinality of $f^{-1}(B)$ for a map $f:M\to N$. In a brief parenthetical, they remark that applying this theory to a tuple map $(f_1,\dots,f_k):M \to N^k$ and letting $B$ be the diagonal in $N^k$ gives a ``Nielsen number of coincidences of two or more maps'' which essentially matches the theory presented in this paper.

The author would like to thank Peter Wong for this reference.

\end{document}